%% file: Degenerate_case_Submit_15_02_2016.tex
\documentclass[reqno,fleqn]{amsart}

\usepackage[utf8]{inputenc}
\usepackage[english]{babel}

\newtheorem{theorem}{Theorem}[section]

\theoremstyle{definition}

\theoremstyle{remark}
\newtheorem{remark}[theorem]{Remark}

\numberwithin{equation}{section}

\usepackage{amssymb}
\usepackage{amsmath}
\usepackage{amsaddr}
\usepackage[usenames]{color}

\input{macro_REP}

\newcommand \A[1]{{\bf (#1)}}

\begin{document}

\title{Stability of densities for perturbed degenerate Diffusions}
\thanks{The article was prepared within the framework of a subsidy granted to the HSE by the Government of the Russian Federation for the implementation of the Global Competitiveness Program.   \\ 
Support by Deutsche Forschungsgemeinschaft through the Research
Training Group RTG 1953 is gratefully acknowledged. }

\author{A. Kozhina}\address{Higher School of Economics, Shabolovka 31, Moscow, Russian Federation and RTG 1953, Institute of Applied Mathematics, Heidelberg University, Germany, aakozhina@hse.ru}

\date{\today} 
\begin{abstract} 
We study the sensitivity of the densities of some Kolmogorov like degenerate diffusion processes with respect to a perturbation of the coefficients of the non-degenerate component. 
 Under suitable (quite  sharp) assumptions  we quantify how the pertubation of the SDE affects the density. 
 Natural applications of these results appear in various fields from  mathematical finance to kinetic models. \\
Keywords: Diffusion Processes, Markov Chains, Parametrix, H\"older Coefficients, bounded drifts.
\end{abstract}

\maketitle

\section{Introduction}

We consider $\R^d \times \R^d-$valued processes that follow the dynamics:

\begin{align}
\begin{cases}
dX_t =  b(X_t,Y_t)dt+  \sigma(X_t, Y_t)dW_t, \\
dY_t = X_t dt,  t\in \lbrack 0,T], 
\end{cases}
\label{1} 
\end{align}
where  $b:  \mathbb{R}^{2d}\rightarrow \mathbb{R}^{d},\  \sigma: \mathbb{R}^{2d}\rightarrow \R^d\otimes \R^d$ are bounded coefficients that are   H\"older continuous in space (this condition will be possibly relaxed for the drift term $b$) and 
$W $ is a Brownian motion on some filtered probability space $(\Omega,\F,(\F_t)_{t\ge 0},\P) $. $T>0$ is a fixed deterministic final time. Also, $a(x,y):=\sigma\sigma^*(x,y)$ is assumed to be uniformly elliptic.
  
  We now introduce a perturbed version of \eqref{1} with dynamics:
\begin{align}
\begin{cases}
dX_t^{(\varepsilon)} = b_\varepsilon(X_t^{(\varepsilon)},Y_t^{(\varepsilon)})dt+  \sigma(X_t^{(\varepsilon)}, Y_t^{(\varepsilon)})dW_t, \\
dY_t^{(\varepsilon)} =  X_t^{(\varepsilon)} dt,
t\in \lbrack 0,T],  \label{2}
\end{cases}
\end{align}
where $b_\varepsilon: \mathbb{R}^{2d}\rightarrow \mathbb{R}^{d},\  \sigma_\varepsilon: \mathbb{R}^{2d}\rightarrow \R^d\otimes \R^d$ 
satisfy at least the same assumptions as $b,\sigma$ and are in some sense meant to be \textit{close} to $b,\sigma $ for small values of $\varepsilon>0$.

In particular those assumptions guarantee that \eqref{1}  admits a unique weak solution, see
e.g. \cite{meno:10}. The unique weak solution of \eqref{1} admits a density $p(t,(x,y),(x',y'))$  for all $t>0$ that satisfies the Aronson like bounds (see \cite{dela:meno:10} and \cite{meno:10}).

Such kind of processes as \eqref{1} appear in various applicative fields. For instance, in mathematical finance, when dealing
with Asian options, $X$ can be associated with the dynamics of the underlying asset and its integral $Y$ is involved in the option Payoff.
Typically, the price of such options writes $\E_x [\psi (X_T , T^{-1} Y_T )]$, where for the put (resp. call) option the function
$\psi(x, y) = (x - y)^+ $ (resp.$ (y - x)^+$), see \cite{baru:poli:vesp:01} and \cite{lape:pard:sent:98}.
It  is, thus, useful to specifically quantify how a perturbation of the coefficients impacts the option prices.

The cross dependence of the dynamics of $X$ in $Y$ is also important when handling kinematic models or Hamiltonian
systems. For a given Hamilton function of the form $ H (x, y) = V (y) + \frac{|x|^2}{2},$ where $V$ is a potential and $\frac{|x|^2}{2}$ the kinetic
energy of a particle with unit mass, the associated stochastic Hamiltonian system would correspond to $ b(X_s, Y_s) =-(\partial_y V (Y_s) + F (X_s, Y_s) X_s) $ in \eqref{1}, where $ F$ is a friction term. 
When $ F > 0$ natural questions arise concerning the
asymptotic behavior of $(X_t, Y_t),$ for instance, the geometric convergence to equilibrium for the Langevin equation is
discussed in Mattingly and Stuart \cite{matt:stua:high:02}, numerical approximations of the invariant measures in Talay \cite{tala:02}, the case
of high degree potential V is investigated in H\'erau and Nier \cite{hera:nier:04}.  

%%% Ремарка про неограниченный тренд - можно оставить так, а можно сослаться на работу с А.Марковой.

The goal of this work is to investigate how the closeness of $(b_\varepsilon,\sigma_\varepsilon)$ and $(b,\sigma)$ is reflected on the respective
densities of the associated processes. 
In many applications (misspecified volatility models or calibration procedures) it can be useful to know how the controls on the differences $|b- b_\varepsilon|, |\sigma - \sigma_\varepsilon|$ (for suitable norms) impact the
difference $p_\varepsilon -p $ of the densities corresponding respectively to the dynamics with the perturbed parameters and
the one of the model.

\subsection{Assumptions and Main Results.}

Let us introduce the following assumptions. Below, the parameter $\varepsilon>0$ is fixed and the constants appearing in the assumptions do not depend on $\varepsilon $.
\begin{trivlist}
\item[(\textbf{A1)}] \textbf{(Boundedness of the coefficients)}. The components of the vector-valued functions $b(x,y), b_\varepsilon(x,y)$ and the
matrix-functions $\sigma(x,y),\sigma_\varepsilon(x,y)$ are bounded measurable. Specifically, there exist constants $K_1,K_2>0$  s.t.
\begin{eqnarray*}
\sup_{(x,y) \in \R^{2d}}|b(x,y)|+\sup_{(x,y) \in \R^{2d}}|b_\varepsilon(x,y)|\le K_1,\\
\sup_{(x,y) \in \R^{2d}}|\sigma(x,y)|+\sup_{(x,y) \in  \R^{2d}}|\sigma_\varepsilon(x,y)|\le K_2.
\end{eqnarray*}

\item[(\textbf{A2})] \textbf{(Uniform Ellipticity)}. The matrices $a:=\sigma\sigma^*, a_\varepsilon:=\sigma_\varepsilon\sigma_\varepsilon^*$ are uniformly elliptic, i.e. there exists $\Lambda \ge 1,\ \forall (x,y,\xi)\in (\R^d)^3$,
\begin{eqnarray*}
\Lambda^{-1} |\xi|^2 %&
\le \langle a(x,y)\xi,\xi\rangle \le \Lambda |\xi|^2,%\\
\Lambda^{-1} |\xi|^2 %&
\le \langle a_\varepsilon(x,y)\xi,\xi\rangle \le \Lambda |\xi|^2.
\end{eqnarray*}

\item[(\textbf{A3})] \textbf{(H\"older continuity in space)}.  For some $ \gamma \in (0,1]$ , $\kappa<\infty $,
\begin{eqnarray*}
\left\vert \sigma(x,y)-\sigma(x',y')\right\vert+\left\vert \sigma_{\varepsilon}(x,y)-\sigma_{\varepsilon}(x',y')\right \vert &\leq &
\kappa\left( \left\vert x-x'\right\vert + \left\vert y-y'\right\vert \right) ^{\gamma }.
\end{eqnarray*}
Observe that the last condition also readily gives, thanks to the boundedness of $\sigma,\sigma_\varepsilon $ that $a,a_\varepsilon $ are also uniformly $\gamma $-H\"older continuous.\\

\end{trivlist}

For a given $\varepsilon>0 $, we say that assumption \textbf{(A)} holds when conditions \textbf{(A1)}-\textbf{(A3)} are in force. Let us now introduce, under \textbf{(A)}, the quantities that will bound the difference of the densities in our main results below. Set for $\varepsilon>0$:
\begin{eqnarray*}
 \forall q\in (1,+\infty],\ 
  \Delta _{\varepsilon,b,q}:=\sup_{t\in[0,T]}\|b(t,.)-b_\varepsilon(t,.)\|_{L^q(\R^d)} .
\end{eqnarray*}
Since $\sigma,\sigma_\varepsilon$ are both $\gamma $-H\"older continuous, see (\textbf{A3}), we also define 
$$\Delta_{\varepsilon,\sigma,\gamma}:=|\sigma(.,.)-\sigma_{\varepsilon}(.,.)|_{\gamma}, $$
where for $\gamma\in (0,1] $, $|.|_\gamma $ stands for the usual H\"older norm in space on $C_b^\gamma(\R^d,\R^d~\otimes~\R^d) $ (space of H\"older continuous bounded functions, see e.g. Krylov \cite{kryl:96})  i.e. :
$$|f|_\gamma:=\sup_{x\in \R^d}|f(x)|+[f]_\gamma,\ [f]_\gamma:= \sup_{x\neq y,(x,y)\in \R^{2d}}\frac{|f(x)-f(y)|}{|x-y|^\gamma}.$$

 The previous control in particular implies for all $((x,y),(x',y'))\in (\R^{2d})^2 $:
$$|a(x,y)-a(x',y') - a_{\varepsilon}(x,y) + a_{\varepsilon}(x',y')| \le 2(K_2+\kappa)\Delta_{\varepsilon,\sigma,\gamma} \left( |x-x'|+|y-y'| \right) ^{\gamma}.$$

We eventually set $\forall q\in (1,+\infty],$
\begin{equation*}
\Delta_{\varepsilon,\gamma, q}:=\Delta_{\varepsilon,\sigma,\gamma}+\Delta_{\varepsilon,b,q},
\end{equation*}
which will be the key quantity governing the error in our results.

We will denote, from now on, by $C$ a constant depending on the parameters appearing in \textbf{(A)} and $T$. We reserve the notation $c$ for constants that only depend on \textbf{(A)} but not on $T$. The values of $C,c$ may change from line to line and do not depend on the considered pertubation parameter $\varepsilon $. 

\begin{THM}[Stability Control]
	\label{MTHM}
	Fix $T>0$. Under (\textbf{A}), for $q\in (4d,+\infty]$, there exists $C:=C(q)\ge 1,c\in (0,1]$ s.t. for all $0<t\le T, ((x,y),(x',y')) \in (\R^{2d})^2$:
	\begin{align*}
		\label{CTR_D}
	|(p-p_\varepsilon)(t,(x,y), (x',y'))| \le C \Delta_{\varepsilon, \gamma, q}    \hat p_{c,K}(t,(x,y),(x',y')),
	\end{align*}
	where $p(t,(x,y), (.,.)), p_\varepsilon(t,(x,y), (.,.)) $ respectively stand for the transition densities  at time $t$ of equations \eqref{1}, \eqref{2} starting from $ (x,y) $ at time $0$.
	Also, we denote for a given $c>0$ and for all $(x',y' )\in  \R^{2d}$, 
\begin{equation}
\label{KOLM_DENS}
\hat p_{c,K}(t,(x,y), (x',y')):=\frac{c^{d} 3^{d/2}}{(2\pi t^2)^{d}} \exp\left( -c \left[ \frac{|x'-x|^2}{4t}+
3 \frac{|y'-y-(x+x')t/2|^2}{t^3} \right] \right),
\end{equation}
which enjoys the semigroup property, i.e. $\forall 0 \le s<t \le T,$
\begin{align*}
\int_{\R^{2d}} \hat p_{c,K} (s,(x,y),(w,z)) \hat p_{c,K} (t-s,(w,z),(x',y')) dw dz = \hat p_{c,K} (t,(x,y),(x',y')).
\end{align*}
The subscript $K$ in the notation $\hat p_{c, K}$ stands for \textit{ Kolmogorov like equations} and $\hat p_{c, K}(t, (x,y),(\cdot , \cdot ))$ denotes the density of $$\left( X_t^{ c, K} \right):=\left( X_t^{K,c,1}, X_t^{K,c,2} \right) = \left(  x+ \frac{\sqrt{2}W_t}{c^{1/2}}, y+\int_0^t X_s ds\right). $$
We refer for details to the seminal paper \cite{kolm:33} and \cite{kona:meno:molc:10}, \cite{dela:meno:10} for further extensions.
\begin{remark}
\label{MULTI_S_REM}
Observe carefully that the density in \eqref{KOLM_DENS} exhibits a \textit{multiscale} behavior. The non degenerate component has the usual diffusive scale in $t^{1/2} $ corresponding to the self-similarity index or typical scale of the Brownian motion at time $t$, whereas the degenerate one has a faster typical behavior in $t^{3/2} $ corresponding to the typical scale of the integral $\int_0^tW_sds, $ 
associated with the parameters $y,y' $.

\end{remark}
\end{THM}

\begin{remark}
Note that the same result could be achieved in the non-homogeneous case without additional assumptions (see \cite{dela:meno:10} for details).
\end{remark}

\begin{remark}
Observe as well that the control of Theorem \ref{MTHM}  should as well hold for the Euler schemes for the degenerate Kolmogorov SDEs introduced in \cite{lema:meno:10} associated with \eqref{1} and \eqref{2}  respectively. See also \cite{kona:kozh:men:15} for the sensitivity of perturbed Markov Chains in the non-degenerate case. 
\end{remark}

\begin{remark}
Let us mention that for applicative purposes, perturbations of the degenerate component could be very interesting as well. 
The first natural perturbation we have in mind would be to consider
$$dX_t^{\varepsilon,2}=\{X_t^{\varepsilon,1}+\varepsilon F(X_t^\varepsilon)\}$$
in \eqref{2} for a smooth bounded function $F$.
However, this setting would require a more subtle handling of the proxy processes involved, in order to make the parametrix approach work. In particular,  similar difficulties than those arising in \cite{dela:meno:10} would occur leading to truncate the parametrix series (because of the non-linear dynamics) and to control the reminders with stochastic control arguments. The investigation of such perturbations will concern further research.  
\end{remark}

The paper is organized as follows. We recall in Section \ref{PARAM_REPRESENT}  some basic facts about parametrix expansions for
the densities of degenerate diffusions. We then detail in Section \ref{STAB_DIFF} how to perform a stability analysis
of the parametrix expansions in order to derive the result of Theorem  \ref{MTHM}. 

\section{Parametrix Representation of the Density}
\label{PARAM_REPRESENT}
From \cite{meno:10} it follows that \eqref{1} has under \A{A} a unique weak solution. We aim at proving that the solution has for each $t\in (0,T]$ a density which can be represented as the sum of a parametrix series.

If the coefficients in \eqref{1} are not smooth, but satisfy \A{A}, it is then possible to use a mollification procedure, 
taking
 $b_\eta(x,y):=b\star \rho_{\eta}(x,y), \ \sigma_\eta(x,y):=\sigma~\star~\rho_{\eta}(x,y),\ x,y \in \R^d $ where $\rho_{\eta} $ is a smooth mollifying kernel and $\star$ stands for a standard convolution operation and $\eta\in[0,1]$, the case $\eta=0 $ corresponding to the initial process in \eqref{1}.

 For mollified coefficients, the existence and smoothness of the density $p_{\eta} $ for 
the associated process $(X^{\eta}_s,Y^{\eta}_s) $ follows from the H\"ormander theorem (see e.g. \cite{horm:67} or \cite{norr:86}). Thus,  we can apply the parametrix technique directly for  $p_{\eta}$.

Roughly speaking, the parametrix approach consists in approximating the process by a proxy which has a known density, here a Gaussian one, and then in investigating the difference through the Kolmogorov equations.  Various approaches to the parametrix expansion exist, see e.g. Il'in \textit{et al.}
\cite{ilin:kala:olei:62}, Friedman \cite{frie:64} and McKean and Singer \cite{mcke:sing:67}. The latter approach will be the one used in this work since it can  be directly extended to the discrete case for Markov chain approximations of equations \eqref{1} and \eqref{2}. Let us mention in this setting the works of Konakov \and Mammen, see\cite{kona:mamm:00}, 
 \cite{kona:mamm:02}.

For the parametrix development we need to introduce a “frozen” diffusion process $  (\tilde X_s^{\eta}, \tilde Y_s^{\eta})_{s \in [0,t] } $% below
. Namely
for fixed $(x',y') \in \R^{2d}, t \in (0,T]$ define for all $s\in [0,t] $:

\begin{align}
\begin{cases}
d \tilde X_s^{t,x',y', \eta} = \sigma_\eta (x',y'-x'(t-s) ) dW_s,   \ \ \tilde X_0^{t,x',y',\eta}=x,\\
d \tilde Y_s^{t,x',y', \eta} =\tilde  X_s^{t,x',y', \eta} ds, \ \   \tilde Y_0^{t,x',y', \eta}=y. 
\end{cases}
\label{1F} 
\end{align}

Observe that for $\eta \in [0,1]$ the above SDE integrates as 
\begin{eqnarray*}
(\tilde X^{t,x',y',\eta}_s, \tilde Y^{t,x',y',\eta}_s)  = R_s \begin{pmatrix} x \\ y \end{pmatrix} +\int_0^s R_{s-u} B \sigma_\eta(x',y'-x'(t-u))dW_u, 
\end{eqnarray*}
where
$R_s=\begin{pmatrix} 1 && 0 \\ s && 1 \end{pmatrix} , \ B=\begin{pmatrix} I_{d\times d} \\  0_{d \times d} \end{pmatrix},$
 which implies that $( \tilde X^{t,x',y',\eta}, \tilde Y^{t,x',y',\eta})$ is a Gaussian process. In particular, its density at time $t$ writes: 
\begin{eqnarray*}
\tilde p_{\eta}^{t, x',y'}(t, (x,y),(x',y'))& 
\\= \frac{1}{(2 \pi)^{d} {\rm det}(C_t)^{1/2}}& \exp \left( - \frac 12 \langle C_t^{-1} ( R_t \begin{pmatrix} x \\ y \end{pmatrix}-\begin{pmatrix} x' \\ y' \end{pmatrix}), R_t \begin{pmatrix} x \\ y \end{pmatrix}-\begin{pmatrix} x' \\ y' \end{pmatrix}\rangle \right),
\end{eqnarray*}
 where $C_t=\int_0^t R_{t-u} B \sigma_\eta \sigma^*_\eta (x',y'-x'(t-u)) B^* R^*_{t-u} du.$
 From this explicit expression, standard Gaussian like computations (see e.g. \cite{kona:meno:molc:10}) imply that there exits $C\ge 1,c\in (0,1]$ s.t.
 $$C^{-1} \hat p_{c^{-1},K}(t, (x,y),(x',y')) \le \tilde p_{\eta}^{t, x',y'}(t, (x,y),(x',y')) \le C \hat p_{c,K} (t, (x,y),(x',y')) .$$
 Also, we have the following controls of the derivatives 
\begin{eqnarray}
\label{DIF_CONTR}
 & \exists C>0,\  \forall \alpha=(\alpha_1,\alpha_2),  \ |\alpha|\le 4, \\
& |D^{\alpha_1}_x D^{\alpha_2}_y \tilde p_{\eta}^{t, x',y'}(t, (x,y),(x',y'))| \le \frac{C}{t^{|\alpha_1|/2 + 3|\alpha_2|/2 }} \hat p_{c, K}(t, (x,y),(x',y')). \notag
\end{eqnarray}
Observe that these controls also reflect the multi scale behavior already mentioned in Remark \ref{MULTI_S_REM}. They are also uniform w.r.t. $\eta \in [0,1] $.

\begin{remark}
The arguments in the second variable of the diffusion coefficient
can seem awkward at first sight,they actually correspond to the transport of the frozen final point $(x',y')$ by the backward differential system:
$
\begin{pmatrix}
 \dot x_s=0 \\ \dot y_s=x_s 
\end{pmatrix},
\begin{pmatrix}
 x_t=x' \\ y_t=y'
\end{pmatrix}.
$
This choice is performed to have a \textit{''compatibility''}  condition in the difference of generators in the parametrix expansion.
See the controls on $H^\eta$ established in \eqref{KER_CONTR} below.
\end{remark}

The processes $(X_s^{\eta},Y_s^{\eta})$  and $(\tilde X_s^{t,x',y', \eta}, \tilde Y_s^{t,x',y', \eta}), \ s \in [0,t],$   have the following generators: $\forall (x,y) \in \R^{2d}, \psi \in C^2(\R^{2d}),$

\begin{equation}
\begin{split}
L^{\eta} \psi(x,y)= \left( \frac{1}{2} Tr \left( a_{\eta}(x,y) D_{x}^{2}\psi \right) + \langle b_{\eta}(x,y), \nabla_x \psi \rangle +\langle x, \nabla_y \psi \rangle \right) (x,y), \\
\tilde L_s^{t,x',y',\eta} \psi (x,y)= \left( \frac{1}{2} Tr \left( a_{\eta}(x',y'-x'(t-s) ) D_x^{2} \psi \right)  + \langle x, \nabla_y \psi \rangle \right) (x,y).
\end{split}
\label{Generators}
\end{equation}

Let us define for notational convinience $\tilde p_\eta (t,(x,y), (x',y')) := \tilde  p_\eta^{t, x',y'}(t,(x,y), (x',y')),$ that is in $ \tilde p_\eta(t,(x,y), (x',y'))$ we consider
the density of the frozen process at the final point and observe it at that specific point.

The density $\tilde p_{\eta} $ readily satisfies the  Kolmogorov Backward equation:
\begin{eqnarray}
\label{KOLM_GEL}
\begin{cases}
\partial_u \tilde p_{\eta}(t-u,(x,y),(x',y'))+\tilde L_u^{t,x',y',\eta} \tilde p_{\eta} (t-u,(x,y),(x',y'))=0,\\ 
	0 < u<t, (x,y),(x',y') \in \R^{2d},\\
\tilde p_{\eta}(t-u,(\cdot, \cdot),(x',y')) \underset{t-u\downarrow 0}{\rightarrow} \delta_{(x',y')}(.).
 \end{cases}
\end{eqnarray}

On the other hand, since  the density of $(X^{\eta}_s,Y^{\eta}_s)$ is smooth, it must satisfy the Kolmogorov forward equation (see e.g. Dynkin \cite{dynk:65}). For a given starting point $(x,y)\in \R^{2d}$ at time $0$,
\begin{eqnarray}
\begin{cases}
\partial_u  p_\eta (u,(x,y),(x',y'))-L^{\eta * } p_\eta (u,(x,y),(x',y'))=0,\ 0< u\le t, (x,y) \in \R^{2d},\\
 p_\eta (u,(x,y),.) \underset{u\downarrow 0}{\rightarrow} \delta_{(x,y)}(.),
 \end{cases} \notag \\
\label{KOLM}
\end{eqnarray}

where $L^{\eta*} $ stands for the adjoint (which is  well defined since the coefficients are smooth) of the generator $L^{\eta}$ in \eqref{Generators}.

Equations \eqref{KOLM_GEL} and \eqref{KOLM} yield the formal expansion below:

\begin{eqnarray}
\label{EQ_PARAM}
(p_{\eta}-\tilde p_{\eta})(t,(x,y),(x',y'))\\
=\int_{0}^{t} du \partial_u \int_{\R^{2d}} dz dw p_{\eta}(u,(x,y),(w,z)) \tilde p_{\eta}(t-u,(w,z),(x',y'))\nonumber\\
=\int_{0}^{t} du \int_{\R^{2d}}dz dw \bigg( \partial_u p_{\eta}(u,(x,y),(w,z)) \tilde p_{\eta}(t-u,(w,z),(x',y')) \nonumber \\
+  p_{\eta}(u,(x,y),(w,z)) \partial_u \tilde p_{\eta}(t-u,(w,z),(x',y')) \bigg)\nonumber \nonumber \\
=\int_{0}^{t} du \int_{\R^{2d}}dz dw \bigg( L^{\eta*}  p_{\eta}(u,(x,y),(w,z)) \tilde p_{\eta}(t-u,(w,z),(x',y')) \nonumber \\
-  p_{\eta}(u,(x,y),(w,z)) \tilde L^{\eta} \tilde p_{\eta}(t-u,(w,z),(x',y')) \bigg)\nonumber \\
=\int_{0}^{t} du \int_{\R^{2d}} dz dw  p_{\eta}(u,(x,y),(w,z))(L^{\eta} -\tilde L^{\eta} )\tilde p_{\eta}(t-u,(w,z),(x',y')),  \notag
\end{eqnarray}

using the Dirac convergence for the first equality, equations \eqref{KOLM_GEL} and \eqref{KOLM} for the third one. We eventually take the adjoint for the last equality. Note carefully that the differentiation under the integral is also here formal since we would need to justify that it can actually be performed using some growth properties of the density and its derivatives which we \textit{a priori} do not know.

Let us now introduce the notation  
$$f\otimes g (t,(x,y),(x',y'))=\int_{0}^{t} du \int_{\R^{2d}} dz dw f(u,(x,y),(w,z)) g(t-u,(w,z),(x',y')) $$ 
for the time-space convolution  We now introduce the \textit{parametrix} kernel:
\begin{equation*}
H^{\eta}(t,(x,y),(x',y')):=(L^{\eta}-\tilde L^{\eta})\tilde p_{\eta}(t,(x,y),(x',y')).
\end{equation*}

\begin{remark}
Note carefully that in the above kernel $H^{\eta}$, because of the linear structure of the degenerate component in the model, the most singular terms, i.e. those involving derivatives w.r.t. $y$, i.e. the \textit{fast} variable, vanish (see Remark \ref{MULTI_S_REM} and \eqref{DIF_CONTR}).
\end{remark}

With those notations equation \eqref{EQ_PARAM} rewrites:
 \begin{eqnarray*}
 (p_{\eta}-\tilde p_{\eta})(t,(x,y),(x',y'))=p_{\eta}\otimes H^{\eta} (t,(x,y),(x',y'))\\
 =\int_0^t du\int_{\R^{2d}} p_\eta(u,(x,y),(w,z))H^\eta(t-u,(w,z),(x',y'))dwdz.
 \end{eqnarray*}

From this expression, the idea then consists in iterating this procedure for $p_{\eta}(u,(x,y),(w,z)) $
in \eqref{EQ_PARAM} introducing the density of a process with frozen characteristics in $(w,z)$ which is here the integration variable. This yields to iterated convolutions of the kernel and leads to the formal expansion:
 \begin{eqnarray}
 \label{DEV_SER}
p_{\eta}(t,(x,y),(x',y'))=\bsum{r=0}^{\infty} \tilde p_{\eta} \otimes H^{{\eta}, (r)}(t,(x,y),(x',y')),
\end{eqnarray}

where $\tilde p_{\eta}\otimes H^{{\eta},(0)}=\tilde p_{\eta}, H^{{\eta},(r)}=H^{\eta}\otimes H^{{\eta},(r-1)},\ r\ge 1$. 

Obtaining estimates on $p_{\eta}$ from the formal expression \eqref{DEV_SER} requires to have good controls on the right-hand side. 

  Precisely thanks to \eqref{DIF_CONTR},  we first get that uniformly in $\eta \in [0,1]$ (thanks to \A{A} and the specific choice of the \textit{freezing} parameters in the proxy), there exist $ c_1>1, \ c \in(0,1]$ s.t. for all $u \in [0,t),$
 \begin{eqnarray}
\label{KER_CONTR}
\left \vert H^\eta(t-u,(w,z),(x',y'))\right \vert  \notag \\
\le \frac{1}{2}Tr \left\{ a_\eta(w,z)-a_\eta(x',y'-x'(t-u)) \right\} D_w^2 \tilde p_\eta (t-u, (w,z),(x',y')) \notag \\
+\langle b_\eta(w,z), D_w \tilde p_\eta(t-u,(w,z),(x',y')) \rangle \notag \\
\le \left[ \frac{C |w-x'|^{\gamma}+|z-y'-x'(t-u)|^{\gamma}}{2(t-u)}+\frac{C K_1}{(t-u)^{1/2}} \right] \hat p_{c,K}(t-u,(w,z),(x',y')) \notag \\
 \le c_1  \left( 1 \vee T^{(1-\gamma) /2} \right) \frac{\hat p_{c,K}(t-u,(w,z),(x',y')) }{(t-u)^{1-\gamma/2}}.
\end{eqnarray}

We can establish by induction the following key result.

\begin{LEMME}
\label{Lemma_bound}
There exist constants $C\ge 1, c\in(0,1] $ s.t. for all $\eta\in [0,1] $ one has for all $(t,(x,y),(x',y'))\in (0,T]\times (\R^{2d})^2 $:
\begin{eqnarray*}
\left \vert \tilde p_\eta \otimes H^{\eta,(r)}(t,(x,y),(x',y')) \right \vert \\
\le C^{r+1} t^{r \gamma/2} B\left ( 1, \frac{\gamma}{2} \right) \times B\left ( 1+\frac{\gamma}{2} , \frac{\gamma}{2} \right) \times \dots \times
B\left ( 1+\frac{(r-1)\gamma}{2}, \frac{\gamma}{2} \right) \notag \\ \times  \hat p_{c, K}(t, (x,y), (x',y')),  \ \ r\in \N^{*}. \notag
\end{eqnarray*}
\end{LEMME}

\begin{proof}

The result \eqref{DIF_CONTR} in particular yields that $\exists C_2>0, \forall u \in (0,t],$ $ \tilde p_\eta(t-u,(x,y),(w,z)) \le C_2 \hat p_{c, K} (t-u,(x,y),(w,z))$ uniformly w.r.t. $\eta \in [0,1].$ \\
Setting $C:=c_1 \left(1 \vee T^{(1-\gamma)/2}\right)  \vee C_2, $ we finally obtain also uniformly in $\eta$
\begin{eqnarray*}
\left \vert \tilde p_\eta \otimes H^\eta(t,(x,y),(x',y')) \right \vert \\
\le \int_0^t du \int_{\R^{2d}} \tilde p_\eta (u,(x,y),(w,z)) |H^\eta(t-u,(w,z),(x',y')) | dwdz, \\
\le \int_0^t du \int_{\R^{2d}} C^2 \hat p_{c,K}(u, (x,y),(w,z)) \frac{1}{(t-u)^{1- \gamma/2}} \hat p_{c,K}(t-u, (w,z),(x',y'))dwdz \\
\le C^2 t^{\gamma/2} B\left ( 1, \frac{\gamma}{2} \right) \hat p_{c, K}(t, (x,y), (x',y')),
\end{eqnarray*}
using the semigroup property of $\hat p_{c,K}$ in the last inequality and where $B(p,q) = \int_0^1 u^{p-1} (1-u)^{q-1}du$ denotes the $\beta-$function. 
By induction in $r$:
 \begin{eqnarray*}
\left \vert \tilde p_\eta \otimes H^{\eta,(r)}(t,(x,y),(x',y')) \right \vert \\
\le C^{r+1} t^{r \gamma/2} B\left ( 1, \frac{\gamma}{2} \right) \times B\left ( 1+\frac{\gamma}{2} , \frac{\gamma}{2} \right) \times \dots \times
B\left ( 1+\frac{(r-1)\gamma}{2}, \frac{\gamma}{2} \right) \notag \\ \times \hat p_{c, K}(t-s, (x,y), (x',y')),  \ \ r\in \N^{*}, \notag
\end{eqnarray*}
which means that the sum of the series \eqref{DEV_SER} is uniformly controlled w.r.t. $\eta \in [0,1].$
 
\end{proof}

These bounds imply that the series representing the  density of the initial process   $p_\eta(t,(x,y),(x',y'))$ could be expressed as:
\begin{eqnarray}
p_\eta(t,(x,y),(x',y')) = \sum_{r=0}^{\infty} \tilde p_\eta \otimes H^{(r)} (t,(x,y),(x',y')). 
\label{PARAM_ETA}
\end{eqnarray}

Lemma \ref{Lemma_bound} readily yields the convergence of the series \eqref{PARAM_ETA} and the following bound uniformly in $\eta \in [0,1]$:
$p_{\eta} (t, (x,y),(x',y')) \le c_1 \hat p_{c, K}(t,(x,y),(x',y')).$ 

From the bounded convergence theorem one can derive that
\begin{eqnarray}
p_\eta(t, (x,y),(x',y')) \underset{\eta \rightarrow 0}{\longrightarrow}  \sum_{r=0}^{\infty} \tilde p \otimes H^{(r)} (t,(x,y),(x',y')):= p(t, (x,y),(x',y')),\notag \\ \label{SUM} 
\end{eqnarray}
 uniformly in $(t,(x,y),(x',y')),$
where $\tilde p(u,(x,y),(w,z)):=\tilde p_0 (u,(x,y),(w,z)) $ and $H^{(r)}(t-u,(w,z),(x',y')):=H^{0,(r)}(t-u,(w,z),(x',y')).$ 

Due to the  uniform convergence in $\eta$ (which implies the uniqueness in law):

\begin{eqnarray*}
\int_{\R^{2d}} f(z,w)p_\eta(t,(x,y),(w,z))dw dz \underset{\eta \rightarrow 0}{\longrightarrow} \int_{\R^{2d}} f(z,w) p(t,(x,y),(w,z))dw dz,
\end{eqnarray*}
for all continious and bounded $f$. The well-posedness of the martingale problem and Theorem 11.1.4 from \cite{stro:vara:79} then give that the process $(X_t,Y_t) $ has the transition density which is exactly the sum of the parametrix series $p(t, (x,y),(x',y')). $ 

Thus, we have proved the below proposition.

\begin{PROP}
\label{EX_DENS_DEG}
Under the sole assumption \A{A},  for $t>0 $, the transition density of the process $(X_t, Y_t)$ solving \eqref{1}
exists and can be written as the series in \eqref{SUM}. 
\end{PROP}

\section{Stability}
\label{STAB_DIFF}
We will now investigate more specifically the sensitivity of the density w.r.t. the coeffcients perturbation through the
difference of the series. From Proposition \ref{EX_DENS_DEG} , for a given fixed parameter  $\varepsilon$, under (A) the densities $p(t, (x,y), (\cdot, \cdot) ),  p_{\varepsilon}(t,(x,y), (\cdot, \cdot))$ at time $t$
of the processes in \eqref{1}, \eqref{2} starting from $(x,y)$ at time $0$ both admit a parametrix expansion of the previous type.

Let us consider the difference between  the two parametrix expansions for \eqref{1} and \eqref{2} in the form \eqref{DEV_SER}:
\begin{eqnarray*}
|p(t, (x,y),(x',y'))-p_\varepsilon(t, (x,y),(x',y'))| \\
\le \sum_{r=0}^{+\infty} | \tilde p \otimes H^{(r)}(t, (x,y),(x',y')) -
\tilde p_{\varepsilon} \otimes H_{\varepsilon}^{(r)}(t, (x,y),(x',y'))|. 
\end{eqnarray*}

Since we consider perturbations of the densities with respect to the non-degenerate component, following the same steps as in \cite{kona:kozh:men:15} one can show that the Lemma below holds:
\begin{LEMME}[Difference of the first terms and their derivatives]
\label{LEM_MT}
There exist $c_1\ge 1,\ c\in (0,1]$ s.t. for all $0<t,\  (x,y),(x',y')\in \R^{2d}$ and all multi-index $\alpha,\ |\alpha|\le 4 $,  
$$|D_{x}^\alpha \tilde p(t,(x,y),(x',y'))-D_{x}^\alpha\tilde p_\varepsilon(t,(x,y),(x',y')) |\le  \frac{c_1 \Delta_{\varepsilon, \sigma, \gamma} \hat p_{c,K}(t,(x,y),(x',y'))}{t^{|\alpha|/2} }. $$ 
\end{LEMME}

\begin{LEMME}[Control of the one-step convolution]
\label{ONE_STEP_C}
 For all $0<t,\  (x,y),(x',y')\in \R^{2d}$:
\begin{equation}
\begin{split}
| \tilde p \otimes H^{(1)}(t,(x,y),(x',y'))-\tilde p_\varepsilon \otimes H^{(1)}_\varepsilon(t,(x,y),(x',y')) |  \\
\le  c_1^2  \Big\{ (1\vee T^{(1-\gamma)/2})^{2} [\Delta_{\varepsilon,\sigma,\gamma} +\I_{q=+\infty}\Delta_{\varepsilon,b,+\infty}] B(1, \frac{\gamma}{2})t^{\frac{\gamma}{2}}
 \\+ \I_{q\in (4d,+\infty)}\Delta_{\varepsilon,b,q}B(\frac12+\alpha(q),\alpha(q))t^{\alpha(q)}\Big\}\hat p_{c,K}(t,(x,y),(x',y'))
 \label{esti2},
  \end{split}
\end{equation}
where $c_1,c$ are as in Lemma \ref{LEM_MT} and for $q\in (4d,+\infty) $ we set $\alpha(q)=\frac{1}{2}-\frac{2d}{q}.$
\end{LEMME}

\begin{proof}
Let us write:
\begin{align}
\label{sum1}
| \tilde p \otimes H^{(1)}(t,(x,y),(x',y'))-\tilde p_\varepsilon \otimes H^{(1)}_\varepsilon(t,(x,y),(x',y')) | \le \notag  \\
|(\tilde p- \tilde p_\varepsilon ) \otimes H(t,(x,y),(x',y')) | +| \tilde p_\varepsilon\otimes \Bigl(H- H_\varepsilon\Bigr) (t,(x,y),(x',y'))|:=I+II .
\end{align}

From Lemma \ref{LEM_MT} and  \eqref{KER_CONTR} we readily get for all $q \in (4d, +\infty]$:
\begin{equation}
\label{CTR_I}
I \le  ((1\vee T^{(1-\gamma)/2})c_1)^2 \Delta_{\varepsilon,\gamma,q}
  \hat p_{c_2,K}(t,(x,y),(x',y'))  B(1, \frac{\gamma}{2})t^{\frac{\gamma}{2}}.
\end{equation}

To estimete $(II)$ let us first consider $H-H_\varepsilon$ more precisely:

\begin{eqnarray}
(H-H_\varepsilon)(t-u,(w,z),(x',y')) \label{DIFF_H_HN_EXP}\\
 = \frac{1}{2}  \Tr\bigg\{ a(w,z) - a(x',y'-x'(t-u) )- a_\varepsilon(w,z) + a_\varepsilon(x',y'-x'(t-u)  \bigg\} \notag \\
 \times  D^2_{w} \tilde p(t-u, (w,z),(x',y')) \notag \\ 
+\frac{1}{2}  \Tr\bigg\{a_\varepsilon (w,z) - a_\varepsilon (x',y'-x'(t-u) ) \bigg\} \bigg[   D_{w}^2 ( \tilde p -  \tilde p_\varepsilon) \bigg](t-u,(w,z),(x',y')) \notag \\ 
 + \langle b(w,z) -b_\varepsilon(w,z), D_{w} \tilde p(t-u,(w,z),(x',y')) \rangle \notag \\
 + \langle b_\varepsilon(w,z), D_{w} (\tilde p -  \tilde p_\varepsilon) (t-u,(w,z),(x',y')) \rangle \notag \\
 := \Bigg( \Delta_\varepsilon^1 H + \Delta_\varepsilon^2 H \Bigg)(t-u,(w,z),(x',y')) \notag \\
 + \langle b(w,z) -b_\varepsilon(w,z), D_{w} \tilde p(t-u,(w,z),(x',y')) \rangle \notag  \\
+ \langle b_\varepsilon(w,z), (D_{w} \tilde p -  D_{w} \tilde p_\varepsilon)(t-u,(w,z),(x',y')) \rangle. \notag
\end{eqnarray}

Since  functions $a(w,z),  a_\varepsilon(w,z)$ are H\"older uniformly continuous and \eqref{DIF_CONTR} holds than:
\begin{eqnarray*}
|\Delta_\varepsilon^1 H|(t-u,(w,z),(x',y')) | \\
\le \frac{c \Delta_{\varepsilon,\gamma, \infty} \left( \left \vert w-x' \right \vert^{\gamma} +\left \vert z-y'+x'(t-u) \right \vert ^{\gamma} \right)  \hat p_{c,K}(t-u,(w,z),(x',y'))}{(t-u) }\\
\le c \Delta_{\varepsilon,\gamma, \infty}  \frac{\hat p_{c_2, K}(t-u,(w,z),(x',y'))}{(t-u)^{1-\gamma/2}}. 
\end{eqnarray*}

From Lemma \ref{LEM_MT} and H\"older uniform continuity  of the function $ a_\varepsilon(x,y)$ it follows:
\begin{eqnarray*}
|\Delta_\varepsilon^2 H|(t-u,(w,z),(x',y')) \\
 \le 
\frac{c \Delta_{\varepsilon,\gamma, \infty} \left( \left \vert w-x' \right \vert^{\gamma} +\left \vert z-y'+x'(t-u) \right \vert ^{\gamma} \right)  \hat p_{c,K}(t-u,(w,z),(x',y'))}{(t-u) }\\
\le c \Delta_{\varepsilon,\gamma, \infty} \frac{\hat p_{\tilde c_2, K}(t-u,(w,z),(x',y')) }{(t-u)^{1-\gamma/2}}.
\end{eqnarray*}

Thus, the fact that $|b(w,z)-b_{\varepsilon}(w,z)|\le c \Delta_{\varepsilon, b, \gamma}$  and \eqref{DIF_CONTR} give the control for $q= +\infty$. Namely,
$$|(H-H_\varepsilon)(t-u,(w,z),(x',y'))| \le \left(1 \vee T^{(1-\gamma)/2} \right) c_1 \Delta_{\varepsilon, \gamma, \infty}
\left[  \frac{\hat p_{c,K}(t-u,(w,z),(x',y')) }{(t-u)^{1-\gamma/2}}\right].$$

For $q \in (4d, +\infty)$ we use H\"older inequality in the  time-space convolution involving the difference of the drifts (last term in \eqref{DIFF_H_HN_EXP}). Set
\begin{eqnarray*}
D(t,(x,y),(x',y')) \\
:=\int_0^t du \int_{\R^{2d}} \tilde p_\varepsilon(u,(x,y),(w,z))  \langle [b_{\varepsilon}(w,z)-b(w,z)] , D_{w} \tilde p (t-u,(w,z)(x',y'))\rangle dw dz.
\end{eqnarray*}
Denoting by $\bar q $ the conjugate of $q$, i.e. $q,\bar q>1, q^{-1}+{\bar q}^{-1}=1 $, we get from \eqref{DIF_CONTR}
 and for $q>d$ that:
\begin{eqnarray*}
|D(t,(x,y),(x',y'))| \le c_1^2 \int_0^t \frac{du}{(t-u)^{1/2}} \|b(.,.)-b_\varepsilon(.,.)\|_{L^q(\R^d)} \\
\times \Big\{ \int_{\R^{2d}}[\hat p_{c,K} (u,(x,y),(w,z)) \hat  p_{c,K}(t-u,(w,z),(x',y'))]^{\bar q} dw dz\Big\}^{1/\bar q} \\
\le c_1^2 \Delta_{\varepsilon,b,q}\int_0^t  \frac{3^{d/q}c^{2d}}{(2\pi)^{2d/q} (c\bar q)^{2d/\bar q}} \\
\times \Big\{ \int_{\R^{2d}} \hat p_{c\bar q,K}(u,(x,y),(w,z)) \hat p_{c\bar q,K} (t-u,(w,z),(x',y')) dw dz \Big\}^{1/\bar q} 
\frac{du}{u^{2d/q}(t-u)^{\frac12+2d/q}} \\
\le c_1^2\left(\frac{\sqrt{3} ct^2}{2\pi}\right)^{d/q}
\bar q^{\frac{d}{\bar q}}
\Delta_{\varepsilon,b,q}\hat p_{c,K}( t,(x,y),(x',y')) \int_0^t \frac{du}{u^{2d/q}(t-u)^{\frac12+2d/q}}.
\end{eqnarray*}

Now, the constraint $4d<q<+\infty$ precisely gives that $\frac12+2d(1-\frac{1}{\bar q})<1$ so that the last integral is well defined. We therefore derive:
\begin{eqnarray*}
|D(t,(x,y),(x',y')| \\
\le c_1^2 t^{\frac12-2d/q}\Delta_{\varepsilon,b,q} \hat p_{c,K} (t,(x,y),(x',y')) B(1-2d/q,\frac12-2d/q).
\end{eqnarray*}
In the case $4d<q<+\infty $, recalling that $\alpha(q)=\frac{1}{2}-\frac{2d}{q} $, we eventually get :
\begin{equation}
\label{CTR_DIFF_FINAL_Q}
\begin{split}
| \tilde p_{\varepsilon}(s,(x,y),(w,z))\otimes \Bigl(H- H_{\varepsilon}\Bigr) (t-u,(w,z),(x',y'))| \\
\le  c_1^2 \hat p_{c,K} (t,(x,y),(x',y')) \{ \Delta_{\varepsilon,b,q}t^{\alpha(q)}B(\frac12+\alpha ( q),\alpha ( q))\\
+ 2\Delta_{\varepsilon,\sigma,\gamma}(1\vee T^{(1-\gamma)/2})t^{\gamma/2}B(1,\gamma/2)\}.
\end{split}
\end{equation}

The statement now follows in whole generality from \eqref{sum1}, \eqref{CTR_I}, \eqref{DIF_CONTR} for $q=\infty$ and \eqref{CTR_DIFF_FINAL_Q} for $4d<q<+\infty $.

\end{proof}

The following Lemma associated with Lemmas \ref{LEM_MT} and \ref{ONE_STEP_C} allows to complete the proof of Theorem \ref{MTHM}.

\begin{LEMME}[Difference of the iterated kernels]
\label{LEMME_DIFF_ITE}
For all $0<t\le T,\ (x,y), (x',y') \in (\R^{2d})^{2} $ and for all $r\in \N$: 
\begin{eqnarray}
\label{ITE}
 |(\tilde p \otimes H^{(r)}-\tilde p_\varepsilon \otimes H_\varepsilon^{(r)})(t,(x,y),(x',y')|  \\
\le C^{r}  r \Delta_{\varepsilon, \gamma, q}\left \{ \frac{t^{\frac{r\gamma}{2}} }{ \Gamma \left( 1+ \frac{r\gamma}{2}  \right) }+
\frac{t^{\frac{(r+2)\gamma}{2}} }{ \Gamma \left( 1+ \frac{(r+2)\gamma}{2}  \right) } 
\right\}  \hat p_{c,K}(t,(x,y),(x',y')). \notag
\end{eqnarray}
\end{LEMME}

\begin{proof}
Observe that Lemmas \ref{LEM_MT} and  \ref{ONE_STEP_C} respectively give \eqref{ITE} for $r=0 $ and $r=1$. Let us assume that it holds for a given $r\in \N^*$ and let us prove it for  $r+1$.

Let us denote for all $r\ge 1,$ \\$ \eta_r(t,(x,y),(x',y')):= |(\tilde p \otimes H^{(r)}-\tilde p_{\varepsilon} \otimes H_\varepsilon^{(r)})(t,(x,y),(x',y'))|.$ Write
\begin{align*}
\eta_{r+1}(t,(x,y),(x',y')) = \left \vert ( \tilde p \otimes H^{(r)}-\tilde p_{\varepsilon} \otimes H^{(r)}_{\varepsilon} ) \otimes H(t,(x,y),(x',y')) \right \vert
\\
+\left \vert \tilde p_{\varepsilon} \otimes H_\varepsilon^{(r)} \otimes (H-H_{\varepsilon})(t,(x,y),(x',y')) \right \vert \\
\le \eta_r \otimes \left \vert H \right \vert (t,(x,y),(x',y')) + \left \vert \tilde p_{\varepsilon} \otimes H^{(r)}_{\varepsilon} \right \vert \otimes \left \vert (H-H_{\varepsilon}) \right \vert (t,(x,y),(x',y')).
\end{align*}
Thus, the induction hypothesis we get the result.
\end{proof}

Theorem 1 now simply follows from the controls of Lemma \ref{LEMME_DIFF_ITE}, the parametrix expansions \eqref{1} and \eqref{2} of
the densities $p,p_{\varepsilon}$ and the asymptotics of the gamma function.

\section*{Acknowledgement}
I would like to thank Valentin Konakov and Stephane Menozzi for the problem statement and fruitful discussion during the preparation of this work.

\bibliographystyle{alpha}
\bibliography{bibli}

\end{document}

%% file: macro_REP.tex
\def\I{{\rm{I\!\!I}}}

\def\F{{\mathcal{ F}}}

\def\R{{\mathbb{R}} }

\def\N{{\mathbb{N}} }
\def\E{{\mathbb{E}}  }

\def\P{{\mathbb{P}}  }

\def\I{{\mathbb{I}}}

\def\Tr{{\rm{Tr}}}
\def\bint#1^#2{\displaystyle{\int_{#1}^{#2}}}
\def\bsum#1^#2{\displaystyle{\sum_{#1}^{#2}}}

\def\xdt_#1{X_#1(\Delta t)}

\newtheorem{THM}{Theorem}
\newtheorem{PROP}{Proposition}
\newtheorem{LEMME}{Lemma}